\documentclass[10pt, a4paper, reqno, oneside]{amsart}

\usepackage{amsmath, amsopn, amsfonts, amsthm, amssymb, amscd, enumerate, multicol, scalefnt}
\usepackage{sansmath}
\usepackage{etex}
\usepackage{relsize}
\usepackage{hyperref}
\usepackage{t1enc}
\usepackage{relsize}
\usepackage[latin1]{inputenc}
\usepackage{graphicx}
\usepackage[all]{xy}
\usepackage{color}
\usepackage{slashed}
\usepackage[mathscr]{euscript}
\usepackage{enumitem}
\setlist[itemize]{noitemsep, topsep=1pt, leftmargin=20pt}
\newcommand\bcdot{\ensuremath{
  \mathchoice
   {\mskip\thinmuskip\lower0.2ex\hbox{\scalebox{1.6}{$\cdot$}}\mskip\thinmuskip}}
   {\mskip\thinmuskip\lower0.2ex\hbox{\scalebox{1.6}{$\cdot$}}\mskip\thinmuskip}
   {\lower0.3ex\hbox{\scalebox{1.2}{$\cdot$}}}
   {\lower0.3ex\hbox{\scalebox{1.2}{$\cdot$}}}
}
\theoremstyle{plain}
\newtheorem{theo}{Theorem}[section]

\newtheorem{prop}[theo]{Proposition}

\theoremstyle{definition}
\newtheorem{rem}[theo]{Remark}
\newtheorem{example}[theo]{Example}
\newtheorem{definition}[theo]{Definition}

\theoremstyle{plain}
\newtheorem{lemma}[theo]{Lemma}
\newtheorem{theorem}[theo]{Theorem}

\theoremstyle{definition}

\newtheorem{remark}[theo]{Remark}
\newtheorem{counterexample}[theo]{Counterexample}

\theoremstyle{plain}

\renewcommand{\=}{:=}

\renewcommand{\a}{\alpha}
\renewcommand{\b}{\beta}

\renewcommand{\d}{\delta}
\newcommand{\e}{\varepsilon}
\newcommand{\f}{\varphi}
\newcommand{\g}{\gamma}
\newcommand{\h}{\eta}

\renewcommand{\l}{\lambda}
\newcommand{\w}{\omega}

\newcommand{\D}{\Delta}

\renewcommand{\L}{\Lambda}

\newcommand{\W}{\Omega}

\newcommand{\bR}{\mathbb{R}}

\newcommand{\fO}{\mathsf{O}}

\newcommand{\cB}{\mathcal{B}}
\newcommand{\cC}{\mathcal{C}}

\newcommand{\eF}{\EuScript{F}}

\newcommand{\eM}{\EuScript{M}}

\newcommand{\eP}{\EuScript{P}}

\newcommand{\p}{\partial}

\newcommand{\td}{\mathtt{d}}

\renewcommand{\square}{\kern1pt\vbox
{\hrule height 0.6pt\hbox{\vrule width 0.6pt\hskip 3pt \vbox{\vskip
6pt}\hskip 3pt\vrule width 0.6pt}\hrule height0.6pt}\kern1pt}
\renewcommand{\=}{\  \raisebox{0.15mm}{:} {=} \ }

\newcommand{\diff}{\operatorname{d\!}}

\DeclareMathOperator\Tr{Tr}

\DeclareMathOperator\Sym{Sym}

\DeclareMathOperator\USC{USC}
\DeclareMathOperator\LSC{LSC}

\DeclareMathOperator\Id{Id}

\DeclareMathOperator\Exp{\operatorname{Exp}}

\newcommand{\inj}{\operatorname{inj}}

\newcommand{\wt}{\widetilde}

\newcommand{\ol}{\overline}
\newcommand{\ul}{\underline}
\newcommand{\zero}{\operatorname{o}}
\def\<#1,#2>{\langle\,#1,\,#2\,\rangle}

\newcommand{\Aac}{\`A}
\newcommand{\eac}{\`e}

\newcommand{\Math}{{\it Mathematica\raise5 pt\hbox{$\scriptscriptstyle \circledR$}7}}
\newcommand{\n}{\nabla}

\newcommand{\beq}{\begin{equation}}
\newcommand{\eeq}{\end{equation}}
\def\<#1,#2>{\langle\,#1,\,#2\,\rangle}
\newcommand{\arr}{\begin{array}{rlll}}
\newcommand{\ea}{\end{array}}
\newcommand{\bea}{\begin{eqnarray}}
\newcommand{\eea}{\end{eqnarray}}
\newcommand{\bean}{\begin{eqnarray*}}
\newcommand{\eean}{\end{eqnarray*}}

\def\sideremark#1{\ifvmode\leavevmode\fi\vadjust{%            The remark
\vbox to0pt{\hbox to 0pt{\hskip\hsize\hskip1em%               will appear only
\vbox{\hsize3cm\tiny\raggedright\pretolerance10000%          on the side
\noindent #1\hfill}\hss}\vbox to8pt{\vfil}\vss}}}%           in 3cm
\newcounter{ssig}
\newcounter{ttig}
\renewcommand{\div}{\operatorname{div}}

\title[Maximum principles on Riemannian manifolds]{A note on the strong maximum principle for fully nonlinear equations on Riemannian manifolds}
\author{Alessandro Goffi and Francesco Pediconi}

\thanks{The first-named author is member of GNAMPA of INdAM and has been partially supported by the Fondazione CaRiPaRo Project ``Nonlinear Partial Differential Equations: Asymptotic Problems and Mean-Field Games''. The second-named author is member of GNSAGA of INdAM and has been supported by the project PRIN 2017 ``Real and Complex Manifolds: Topology, Geometry and holomorphic dynamics'' (code 2017JZ2SW5).}

\begin{document}
\begin{abstract} We investigate strong maximum (and minimum) principles for fully nonlinear second order equations on Riemannian manifolds that are non-totally degenerate and satisfy appropriate scaling conditions. Our results apply to a large class of nonlinear operators, among which Pucci's extremal operators, some singular operators like those modeled on the $p$- and $\infty$-Laplacian, and mean curvature type problems. As a byproduct, we establish new strong comparison principles for some second order uniformly elliptic problems when the manifold has nonnegative sectional curvature.
\end{abstract}

\maketitle

%\tableofcontents

\section{Introduction}
This paper is devoted to analyze strong maximum and comparison principles for viscosity solutions to fully nonlinear second order equations on (finite dimensional) Riemannian manifolds $(M^n,g)$ of the general form \beq F(x,u,Du,D^2u)=0 \quad \text{ in $\W \subset M$ ,}  \label{1}\eeq being $\W$ a connected open subset and $F: J^2M \to \bR$ proper, namely non-decreasing in the second entry and non-increasing in the last entry. Here, $J^2M$ denotes the $2$-jet bundle over $M$ (see Section \ref{sec;notation}). More precisely, we are concerned with proving the following main results:
\begin{itemize}[leftmargin=40pt]
\item[(SMP)] Any upper semicontinuous viscosity subsolution to \eqref{1} attaining a nonnegative interior maximum is constant (see Theorem \ref{SMP});
\item[(SmP)] Any  lower semicontinuous viscosity supersolution to \eqref{1} attaining a nonpositive interior minimum is constant (see Theorem \ref{SmP});
\end{itemize}
and one of their consequences, namely the following tangency principle:
\begin{itemize}[leftmargin=40pt]
\item[(SCP)] Let $u,v$ be respectively a viscosity sub- and supersolution to \eqref{1} such that $u\leq v$ in $\Omega$ and $u(x_0)=v(x_0)$ at some $x_0 \in \W$, then $u\equiv v$ in the whole $\Omega$ (see Proposition \ref{SCP2}).
\end{itemize} \smallskip

The first aim of this note is to complete and extend well-known results valid in the classical Euclidean case to the more general realm of Riemannian manifolds. Our second motivation is then to lay the groundwork for investigating (one-side) Liouville-type results for fully nonlinear elliptic problems as \eqref{1} on general Riemannian structures. This would be mainly inspired by the recent nonlinear studies in \cite{BC}, which make use of (SMP) and (SmP), and the linear Liouville properties given in \cite{Gry}, which are intimately connected with the stochastic completeness of the manifold (cf \cite[Sect 13.2]{Gry}), see also \cite{MV,Punzo}. These latter properties will be matter of a future research by the authors and we believe this note would be a starting point to address these issues. 
Before stating our main results, we begin with a glimpse on the literature on maximum principles, starting with (SMP)-(SmP) and concluding with (SCP). \smallskip

The (SMP) and (SmP) for linear equations $$Lu=-\Tr(A(x)D^2u)+b(x)\cdot \n u+c(x)u=0 \quad \text{ in $\W \subset \bR^n$ } $$ with $A$ uniformly elliptic, $b$ bounded and continuous and $c$ nonnegative and bounded, dates back to E. Hopf and it is a consequence of his Boundary Point Lemma (see e.g. \cite{GT,PW,PS}). A refinement of this procedure was implemented by E. Calabi for semicontinuous viscosity solutions to linear equations \cite{Calabi}. Instead, the literature on quasi-linear problems in $\mathbb{R}^n$, mostly modeled on $p$-Laplacian operators, is huge and we refer the interested reader to \cite{PS} for a comprehensive exposition and \cite[Thm 8.1]{PS2} for maximum principles under mild assumptions on the quasi-linear operator. \smallskip

Let us now mention related results for fully nonlinear problems close to our model PDEs. The (SMP) and (SmP) were found by Caffarelli and Cabr\'e \cite[Prop 4.9]{CC} as a consequence of the weak Harnack inequality. Under lower ellipticity conditions on the nonlinear operator $F$, they were derived in \cite{KK} and \cite{BDL1}, the latter covering even many examples of quasi-linear equations. Descriptions via control theoretic and probabilistic arguments of the propagation sets of maxima (and minima) for Hamilton-Jacobi-Bellman equations were provided in \cite{BDL2,BDL3}. Later results for fully nonlinear uniformly elliptic equations with linear gradient growth can be found in \cite[Thm 5.1]{IY}, see also \cite{CLN} for other maximum principles for nonlinear elliptic operators. Finally, recent contributions for fully nonlinear PDEs over H\"ormander vector fields are given in \cite{BG1}. More recent progresses have been made on Hessian operators (mostly for truncated Laplacians), and they can be found in \cite{BGI1,BGI2}.

The literature of the corresponding equations on Riemannian manifolds is poor, although (SMP) and (SmP) are of local nature and one does not usually expect their failure depending on the properties of the considered manifold. For linear elliptic problems, the first result appeared in \cite{Calabi}, see also \cite[Prop 6.84]{Besse} and \cite[Appx A]{PPS} for weak solutions to the manifold Laplace equation.
The case of nonlinear equations modeled on the $p$-Laplacian perturbed by zeroth order terms has been addressed in \cite[Sect 8.5]{PS}. A more general treatment of quasi-linear equations, involving even mean-curvature operators, can be found in \cite[Sect 4.2]{BMPR} under the name of finite maximum principle and in \cite[Thm 3.10]{AMR} for PDEs perturbed by first order terms, see also \cite{PRS} and the references therein. In the context of fully nonlinear operators, the only available results we are aware of in the context of Riemannian manifolds are those established by Harvey and Lawson \cite{HLsmp}, that however cover the case of Hessian-type equations only. 

Our approach is inspired by the landmark papers \cite{BDL1,BDL2,BDL3} and borrows several (viscosity) techniques from these works. In particular, we stress that the procedure implemented here is structured in the spirit of that used by Calabi \cite{Calabi} and the nonlinear one by Bardi and Da Lio \cite{BDL1}, where solutions are meant in the weak viscosity sense, unlike the distributional framework developed in \cite{AMP,PS}. Our main result (SMP) that we prove in Theorem \ref{SMP} for \eqref{1} is based on a classical barrier-type argument for viscosity solutions and works for those operators $F$ that are elliptic according to the following notion $$\sup_{\a>0}F(x,0,q,Q-\a\,q{\otimes}q)>0 \quad \text{ for any } x\in\W \,\, , \,\,\, q\in T^*_xM , \, q\neq0 \,\, , \,\,\, Q\in\Sym^2(T_xM) $$ and fulfilling a scaling condition like $$F\big(x, c s, c q, cQ)\geq \eta(c)F\big(x, s, q, Q) \quad \text{ for any } s \in [-1,0] \text{ and } c \in (0,1] \,\, ,$$ for some function $\eta>0$ (see Section \ref{sec;notation} for detailed assumptions). In particular, we highlight that the ellipticity condition reduces to the non-degeneracy property identified by Bardi and Da Lio \cite{BDL1} in the Euclidean setting whenever the manifold has nonnegative sectional curvature (see Remark \ref{SMPsimp}). In Section \ref{URiem} and Section \ref{sec;SMP} we list several examples to which our main Theorem \ref{SMP} applies, among which equations driven by Pucci's, $p$-Laplace Beltrami, $\infty$-Laplacian and mean curvature operators, together with various singular nonlinear operators modeled on the $p$- and $\infty$-Laplacian for which the viscosity theory on Riemannian manifolds has been developed, see e.g. \cite{Kim}. We will also provide a weak version of the Hopf boundary Lemma for viscosity solutions in Theorem \ref{hopf}. Furthermore, symmetric ellipticity and scaling conditions lead to a version of the (SmP), that is stated in Theorem \ref{SmP} for reader's convenience. \smallskip

Let us now pass to discuss property (SCP), for which the literature is less wide. For classical solutions to linear equations it is a straightforward byproduct of the (SMP). In the case of nonlinear problems, even degenerate, some additional conditions on the operator are needed. Property (SCP) for smooth solutions to quasi-linear equations can be found in \cite[Thm 2.2.2]{PS} under the name of tangency principle. In the fully nonlinear case, strong comparison principles have been addressed in \cite[Rem 3]{BDL1} (when one of the functions is $\cC^2$ via the arguments in \cite{KK}), by Ishii-Yoshimura in \cite[Thm 5.3]{IY} for second order uniformly elliptic equations with Lipschitz growth in $(u,Du)$, N.S. Trudinger \cite{Tru} for Lipschitz continuous viscosity solutions, Y. Giga and M. Onhuma \cite{GO} for semicontinuous solutions, see also \cite[Thm 3.1]{Patrizi} for other results on uniformly elliptic operators and  \cite{OS} for the case of mean-curvature equations, all of them in the Euclidean setting. More recently, strong comparison principles have been proved in \cite{BG1} for degenerate Hamilton-Jacobi-Bellman equations, covering also problems structured over H\"ormander vector fields, see also \cite{LiWang} for tangency principles results for fully nonlinear problems arising in conformal geometry.

In Section \ref{sec;SCP}, we first set up a simple proof for Pucci's extremal operators when one of the involved functions is smooth (see Lemma \ref{SCP1}) and then we prove (SCP) following a strategy implemented in \cite{IY} based on a combination of the (SMP) with the weak comparison principle in \cite{Azagra} that yields the result in the uniformly elliptic case. In particular, we assume $F$ intrinsically uniformly continuous, uniformly elliptic in the sense that
\[
\l \Tr(P) \leq F(x,s,q,Q)-F(x,s,q,Q+P) \leq \L \Tr(P)
\]
for any $x \in M$, $(s,q,Q) \in \wt{J}^2_xM$, $P \in \Sym^2(T_xM)$ with $P \geq 0$ and ellipticity constants $0<\l \leq \L$, and satisfying
\[
|F(x,s_1,q_1,Q)-F(x,s_2,q_2,Q)|\leq C\,\big(|s_1-s_2|+|q_1-q_2|\big)
\]
for any $x \in M$ and $(s_1,q_1,Q),(s_2,q_2,Q) \in J^2_xM$, and prove (SCP) whenever the manifold has nonnegative sectional curvature, see Remark \ref{seccurv} for further comments on this restriction. This latter result is contained in Proposition \ref{SCP2} and essentially relies on proving that the difference $w=u-v$ is a subsolution to a sort of linearized equation, following the path of \cite{IY}.

As far as the weak comparison principle is concerned, the literature is huge when $F$ is strictly proper since the results encompass even first order problems \cite{BCD,CIL}. In the more general case of Riemannian manifold, the only available contributions we are aware of in the second order case have been obtained in \cite{Azagra}. In particular, when the manifold has nonnegative sectional curvature, weak comparison principles are obtained via the Riemannian counterpart of the Euclidean theorem on sums \cite{CIL}, see \cite{Azagra}. However, when $(M,g)$ has negative curvature, the weak comparison continues to hold under a further uniform continuity assumption (see \cite[cond (2b) on Cor 4.10]{Azagra} and Remark \ref{seccurv} below). We refer also to \cite[Appx A]{MariPessoa} for further comparison principles on Riemannian manifolds for quasi-linear problems.

Under the mere properness of the operator one needs some form of ellipticity and the minimal conditions seem to be an open problem, see \cite{BB,Je2,KK,KK2} and the more recent \cite{BG1,BM}.

We finally mention that other kind of maximum principles (at infinity) on Riemannian manifolds can be found in the monographs \cite{Grybook,PRS}, while Alexandroff-Bakelman-Pucci estimates together with maximum principles for subsolutions of linear problems were obtained in \cite{C}. \smallskip

\noindent {\it Plan of the paper}. Section \ref{sec;prel} is devoted to some preliminaries on Differential Geometry. Section \ref{sec;fully} introduces fully nonlinear second order equations on Riemannian manifolds together with the ellipticity definitions used throughout the paper. Section \ref{sec;SMP} comprehends the proof of the (SMP) and (SmP) with a list of prototype examples, while Section \ref{sec;SCP} applies the results to deduce (SCP). \smallskip

\noindent{\it Acknowledgement.} We thank Prof. Hitoshi Ishii for providing us with a copy of \cite{IY}. \medskip

\section{Preliminaries of Riemannian Geometry} \label{sec;prel}

Let $(M,g)$ be a $\cC^3$-Riemannian $n$-manifold, i.e. a connected smooth manifold $M$ of dimension $n$ endowed with a symmetric $(0,2)$-tensor field $g \in \cC^3(M,\Sym^2(T^*M))$ which induces at each point $x \in M$ a positive definite inner product $g_x$ on the tangent space $T_xM$. We denote by $\diff\,$ the exterior derivative on $M$ and by $D$ the Levi-Civita covariant derivative on $(M,g)$. We call {\it domain} any connected open subset $\W \subset M$. If the closure $\ol{\W} \subset M$ is compact, we say that $\W$ has compact closure. 

From now on, all the manifolds are assumed to be connected, all the Riemannian metrics are assumed to be of class $\cC^3$ and they are not necessarily complete. \smallskip

Let $\W \subset M$ be a domain.  Given a function $\f \in \cC^2(\W,\bR)$, it holds $D\f=\diff \f$. Moreover, we denote by $\n \f \in \cC^1(\W,TM|_{\W})$ the {\it gradient of $\f$}, which is defined by $g(\n \f,X)\=\diff \f(X)$ for any $X \in \cC^{\infty}(\W,TM|_{\W})$, and by $D^2 \f \in \cC^0(\W,\Sym^2(T^*M)|_{\W})$ the {\it Hessian of $\f$}, i.e. $D^2\f \= D(\diff \f)$. From the very definition, it follows that $$D^2\f(X,Y) = g(D_X(\n \f),Y) = X(\diff\f(Y)) - \diff\f(D_XY)$$ for any pair of smooth vector fields $X,Y  \in \cC^{\infty}(\W,TM|_{\W})$. \smallskip

The Riemannian manifold $(M,g)$ is in particular a separable, locally compact length space by means of the induced Riemannian distance $\td:M\times M \to \bR$ (see e.g. \cite[Sect 2]{Ped}). We will denote by $B(x,r)$ the open metric ball centered at $x \in M$ of radius $r>0$ in $M$, by $S(x,r) \=\p B(x,r)$ its boundary and by $\ol{B}(x,r) \= \ol{B(x,r)}$ its closure. Notice that, by the classical Hopf-Rinow Theorem, if $(M,g)$ is complete, then a domain has compact closure if and only if it is bounded (see e.g. \cite[Thm 2.5.28]{BBI}). Moreover, for any $x \in M$, we denote by $\Exp_{x}$ the Riemannian exponential map at $x$ and by $\inj(x)$ the injectivity radius of $(M,g)$ at $x$.

Let $x_{\zero} \in M$ be a point. We denote by $|\cdot|$ the norm induced by $g$ on the tangent bundle $TM$ and by $B(0_{x_{\zero}},\inj(x_{\zero})) \subset T_{x_{\zero}}M$ the ball centered at the origin of the tangent space $T_{x_{\zero}}M$ with radius $\inj(x_{\zero})$. Then, it is known that the restricted map $$\Exp_{x_{\zero}}: B(0_{x_{\zero}},\inj(x_{\zero})) \subset T_{x_{\zero}}M \to B(x_{\zero},\inj(x_{\zero})) \subset M$$ is a diffeomorphism of class $\cC^2$ and the function \beq f_{x_{\zero}}:B(x_{\zero},\inj(x_{\zero})) \to \bR \,\, , \quad f_{x_{\zero}}(x)\= \tfrac12\td(x,x_{\zero})^2 \label{f_x}\eeq is of class $\cC^2$ and regular, that is $(\n f_{x_{\zero}})(x)\neq0_x$ for any $x \in B(x_{\zero},\inj(x_{\zero}))$. This follows from the Gauss Lemma, which implies that $\td(x,x_{\zero})=\big|(\Exp_{x_{\zero}})^{-1}(x)\big|$ for any $x \in B(x_{\zero},\inj(x_{\zero}))$.

We recall that the injectivity radius $x \mapsto \inj(x)$ is continuous whenever $g$ is complete (see e.g. \cite{Sakai}). Moreover, for any $0<r<\inj(x_{\zero})$, the ball $B(x_{\zero},r)$ has compact closure, the sphere $S(x_{\zero},r)=(f_{x_{\zero}})^{-1}(\frac{r^2}2)$ is an embedded $(n{-}1)$-submanifold of class $\cC^2$ and, for any $x \in S(x_{\zero},r)$, the gradient $(\n f_{x_{\zero}})(x)$ is outward-pointing perpendicular to the tangent space $T_xS(x_{\zero},r)$. For the sake of notation, given $x,y \in M$ such that $\td(x,y) < \min\{\inj(x),\inj(y)\}$, we denote by $L_{xy}:T_xM \to T_yM$ the parallel transport along the unique minimizing geodesic from $x$ to $y$. \smallskip

Let now $x_{\zero} \in M$ and $0 < r < \inj(x_{\zero})$. We recall that the ball $B(x_{\zero},r)$ is said to be {\it strongly geodesically convex} if, for any ball $B(y,r') \subset B(x_{\zero},r)$, any two points in $B(y,r')$ are joined by a unique minimizing geodesic which is entirely contained in $B(y,r')$. In this case, it holds that $(D^2f_{x_{\zero}})(x)>0$ for any $x \in B(x_{\zero},r)$. The value \beq R(x_{\zero}) \= \sup\big\{0<r<\inj(x_{\zero}) : B(x_{\zero},r) \text{ is strongly geodesically convex } \big\} \label{R(x)} \eeq is called {\it convexity radius of $(M,g)$ at $x_{\zero}$.} We recall that the function $x \mapsto R(x)$ is positive and, from the very definition, it is also $1$-Lipschitz (see e.g. \cite{Xu18}). More precisely: \begin{itemize}
\item[$\bcdot$] if there exists $\tilde{x} \in M$ such that $R(\tilde{x})=+\infty$, then $R(x)=+\infty$ for any $x \in M$;
\item[$\bcdot$] if there exists $\tilde{x} \in M$ such that $R(\tilde{x})<+\infty$, then $R(x) < \infty$ for any $x \in M$ and $$|R(x)-R(y)|\leq \td(x,y) \quad \text{ for any $x,y \in M$ } \, .$$
\end{itemize}

Notice that in the Euclidean case, i.e. when $(M,g)=(\bR^n,g_{\rm flat})$, for any $x_{\zero} \in \bR^n$ it holds that $$R(x_{\zero})=\inj(x_{\zero})=+\infty \,\, , \quad f_{x_{\zero}}(x)=\tfrac12|x-x_{\zero}|^2 \,\, , \quad (\n f_{x_{\zero}})(x)=x-x_{\zero} \,\, , \quad (D^2f_{x_{\zero}})(x)=\Id_{\bR^n}$$ and any parallel transport $L_{xy}: \bR^n \simeq T_x\bR^n \to T_y\bR^n \simeq \bR^n$ coincides with the identity map.

\medskip

\section{Fully nonlinear PDEs on Riemannian manifolds}\label{sec;fully}

\subsection{Notation} \label{sec;notation} \hfill \par

Let $(M,g)$  be a Riemannian manifold and $\W \subset M$ a domain. We denote by $J^2M \to M$ the 2-jet bundle on $M$, which splits as $$J^2M = (M \times \bR) \oplus_{\mathsmaller{M}} T^*M \oplus_{\mathsmaller{M}} \Sym^2(T^*M)$$ by means of the Riemannian metric $g$. For any function $\f \in \cC^2(\W,\bR)$, its 2-jet at $x$ is given by $$\big(j^2\f\big)(x) \= \big(\f(x),D\f(x),D^2\f(x)\big) \in J^2_xM \,\, , \quad J^2_xM= \bR \oplus T_x^*M \oplus \Sym^2(T^*_xM) \,\, .$$ We define the subbundle $\wt{J}^2M \subset J^2M$ by setting $$\wt{J}^2M \= \bigsqcup_{x \in M}\wt{J}^2_xM \,\, , \quad \wt{J}^2_xM \= \big\{(s,q,Q) \in J^2_x(M) : q \neq 0\big\} \,\, .$$ Let us consider a function $F: \wt{J}^2(M) \to \bR$, which is assumed to satisfy the following fundamental condition: \begin{enumerate}[label=(p.), leftmargin=30pt]
\item\label{(p.)} $F$ is {\it proper}, i.e. $F(x,r,q,P)\leq F(x,s,q,Q)$ whenever $r \leq s$ and $Q \leq P$.
\end{enumerate} For the sake of shortness, for any function $\f \in \cC^2(\W, \bR)$ we define $$F[\f](x) \= F\big(x,\big(j^2\f\big)(x)\big) = F(x,\f(x),D\f(x),D^2\f(x)) \quad \text{ for any $x \in \W$ such that $D\f(x) \neq 0$ }\, . $$ Moreover, we set $$\begin{aligned}
\USC(\W) &\= \{ \text{ upper semicontinuous functions $u: \W \to [-\infty,+\infty)$ } \} \,\, , \\
\LSC(\W) &\= \{ \text{ lower semicontinuous functions $v: \W \to (-\infty,+\infty]$ } \}
\end{aligned}$$ and we define $\USC(\ol{\W})$, $\LSC(\ol{\W})$ similarly. Then, we recall the definitions of viscosity sub- and supersolution of the equation \beq F[u](x)=F(x,u(x),Du(x),D^2u(x))=0 \quad \text{ on }\, \W \,\, .\label{PDE}\eeq 

\begin{definition} A function $u \in \USC(\W)$ (resp. $v \in \LSC(\W)$) is a {\it viscosity subsolution (resp. viscosity supersolution) to \eqref{PDE}} if, for any function $\f \in \cC^2(\W,\bR)$ and $x \in \W$ local maximum of $u-\f$ (resp. local minimum of $v-\f$) with $D\f(x)\neq0$, it holds $F(x,u(x),D\f(x), D^2\f(x)) \leq 0$ (resp. $F(x,v(x),D\f(x), D^2\f(x)) \geq 0$). A continuous function $u: \W \to \bR$ is a {\it viscosity solution to \eqref{PDE}} if it is both a viscosity subsolution and a viscosity supersolution. \end{definition}

\begin{rem} Notice that, since we allow $F$ to be singular at $q=0$, this definition is slightly weaker than the one in \cite[Sect 2]{CIL}. Other refined notions of viscosity inequalities via the envelopes has appeared when dealing with singular equations (we refer e.g. to \cite[Def 2.3.1]{Giga}, \cite{Kim} and references therein). However, we stress that if $F$ is continuous and $u: \W \to \bR$ is of class $\cC^2$ with $Du\neq0$, then $u$ is a viscosity subsolution (resp. viscosity supersolution) to $F[u]=0$ if and only if it is a classical subsolution (resp. classical supersolution) to $F[u]=0$. \end{rem}

We introduce now two notions of ellipticity which will be of central importance in this work (see also \cite[Sec 1]{BDL1}). These are: \begin{enumerate}[label=(l.p.e.), leftmargin=40pt]
\item\label{(l.p.e.)} $F$ is {\it lower partially elliptic} if for any $x_{\zero} \in M$ there exists a function $\a_{x_{\zero}} : B(x_{\zero},R(x_{\zero})) \to [0,+\infty)$ such that $$F\big(x,0,(Df_{x_{\zero}})(x),(D^2f_{x_{\zero}})(x){-}\a\,(Df_{x_{\zero}})(x){\otimes}(Df_{x_{\zero}})(x)\big)>0$$ for any $x \in B(x_{\zero},R(x_{\zero}))$ and $\a>\a_{x_{\zero}}(x)$; 
\end{enumerate} \begin{enumerate}[label=(u.p.e.), leftmargin=40pt]
\item\label{(u.p.e.)} $F$ is {\it upper partially elliptic} if for any $x_{\zero} \in M$ there exists a function $\a_{x_{\zero}} : B(x_{\zero},R(x_{\zero})) \to [0,+\infty)$ such that $$F\big(x,0,(Df_{x_{\zero}})(x),\a\,(Df_{x_{\zero}})(x){\otimes}(Df_{x_{\zero}})(x){-}(D^2f_{x_{\zero}})(x)\big)<0$$ for any $x \in B(x_{\zero},R(x_{\zero}))$ and $\a>\a_{x_{\zero}}(x)$; 
\end{enumerate} Here, notice that $R(x_{\zero})$ and $f_{x_{\zero}}$ have been defined in \eqref{R(x)} and \eqref{f_x}, respectively. For the sake of shortness, we say that $F$ is {\it partially elliptic} if it is both lower and upper partially elliptic. Furthermore, we recall that $F$ is {\it uniformly elliptic} if there exist two constants $0<\l \leq \L$ such that $$\l\Tr(P) \leq F(x,s,q,Q)-F(x,s,q,Q+P) \leq \L \Tr(P)$$ for any $x \in M$, $(s,q,Q) \in \wt{J}^2_xM$, $P \in \Sym^2(T_xM)$ with $P \geq 0$. Notice that uniform ellipticity implies partial ellipticity (see Lemma \ref{uepe}), but the converse assertions do not hold true. \smallskip

We also introduce two scaling properties which will play a role (see again \cite[Sec 1]{BDL1}). These are: \begin{enumerate}[label=(l.s.p.), leftmargin=40pt]
\item\label{(l.s.p.)} $F$ has the {\it lower scaling property} if for any $x_{\zero} \in M$ there exist two functions $\h_{x_{\zero}}: (0,1] \to (0,+\infty)$ and $\a_{x_{\zero}} : B(x_{\zero},R(x_{\zero})) \to [0,+\infty)$ such that $$\begin{aligned}
F\big(x,cs,c(Df_{x_{\zero}})(x),&c\big((D^2f_{x_{\zero}})(x){-}\a\,(Df_{x_{\zero}})(x){\otimes}(Df_{x_{\zero}})(x)\big)\big) \geq \\
&\h_{x_{\zero}}(c)F\big(x,s,(Df_{x_{\zero}})(x),(D^2f_{x_{\zero}})(x){-}\a\,(Df_{x_{\zero}})(x){\otimes}(Df_{x_{\zero}})(x)\big)
\end{aligned}$$ for any $c \in (0,1]$, $s \in [-1,0]$, $x \in B(x_{\zero},R(x_{\zero}))$ and $\a>\a_{x_{\zero}}(x)$;
\end{enumerate} \begin{enumerate}[label=(u.s.p.), leftmargin=40pt]
\item\label{(u.s.p.)} $F$ has the {\it upper scaling property} if for any $x_{\zero} \in M$ there exist two functions $\h_{x_{\zero}}: (0,1] \to (0,+\infty)$ and $\a_{x_{\zero}} : B(x_{\zero},R(x_{\zero})) \to [0,+\infty)$ such that $$\begin{aligned}
F\big(x,cs,c(Df_{x_{\zero}})(x),&c\big(\a\,(Df_{x_{\zero}})(x){\otimes}(Df_{x_{\zero}})(x){-}(D^2f_{x_{\zero}})(x)\big)\big) \leq \\
&\h_{x_{\zero}}(c)F\big(x,s,(Df_{x_{\zero}})(x),\a\,(Df_{x_{\zero}})(x){\otimes}(Df_{x_{\zero}})(x){-}(D^2f_{x_{\zero}})(x)\big)
\end{aligned}$$ for any $c \in (0,1]$, $s \in [0,1]$, $x \in B(x_{\zero},R(x_{\zero}))$ and $\a>\a_{x_{\zero}}(x)$.
\end{enumerate} Finally, we recall that $F$ is said to be {\it positively $h$-homogeneous}, with $h \in \bR$, if $$F(x,cs,cq,cQ)= c^h F(x,s,q,Q)$$ for any $c>0$, $x \in M$, $(s,q,Q) \in \wt{J}^2_xM$. Clearly, positive $h$-homogeneity implies both the lower and the upper scaling properties, but the converse assertions do not hold true.

\subsection{Universal Riemannian operators} \label{URiem} \hfill \par

We recall here the construction of a distinguished kind of PDE's on Riemannian manifolds. More specifically, these have constant coefficients and are obtained from some Euclidean operators via the action of the orthogonal group $\fO(n)$ (see \cite[Sect 5]{HL11}). \smallskip

Let $(M,g)$ be a Riemannian manifold. Consider a function $$\ul{\eF} : \bR {\times} \big(\bR^n {\setminus} \{0\}\big) {\times} \Sym(n) \to \bR$$ and assume that it is invariant under the standard left action of the orthogonal group $\fO(n)$, i.e. $$\ul{\eF}(s,v,A) = \ul{\eF}(s,a.v,a.A.a^{\mathsmaller{T}}) \quad \text{ for any } a \in \fO(n) \,\, .$$ Then, one can construct an associated operator $F: \wt{J}^2M \to \bR$ in the following way. Firstly, for any $x \in M$ and for any orthonormal frame $e=(e_1,{\dots},e_n)$ for $T_xM$, we consider the linear isomorphism $$\Phi_{(x,e)}: J^2_xM \to \bR \oplus \bR^n \oplus \Sym(n) \,\, , \quad \Phi_{(x,e)}(s,q,Q) \= (s,\Phi_{(x,e)}^1(q),\Phi^2_{(x,e)}(Q))$$ defined by $$\Phi_{(x,e)}^1(q) \= \big(q(e_1),{\dots},q(e_n)\big)^{\mathsmaller{T}} \,\, , \quad \Phi_{(x,e)}^2(Q) \= \big(Q(e_j,e_{\ell})\d^{\ell i}\big)^{1 \leq i \leq n}_{1\leq j \leq n} \,\, .$$ Then, we define $$\eF: \wt{J}^2M \to \bR \,\, , \quad \eF(x,s,q,Q) \= \ul{\eF}\big((\Phi_{(x,e)})(s,q,Q)\big) \,\, ,$$ where $e=(e_1,{\dots},e_n)$ is any orthonormal frame for the tangent space $T_xM$. Then, by means of the $\fO(n)$-invariance of $\ul{\eF}$, one can easily prove that $\eF$ is well defined. The function $\eF$ is called {\it universal Riemannian operator associated to $\ul{\eF}$}. It is straightforward to check that the following statements hold true: \begin{itemize}
\item[i)] $\eF$ is proper / continuous / uniformly elliptic / positively $h$-homogeneous if and only if $\ul{\eF}$ is too;
\item[ii)] $\eF$ is lower partially elliptic if $$\sup_{\mathsmaller{\a>0}}\,\ul{\eF}(0,v,A-\a\, v{\otimes}v^{\mathsmaller{T}})>0 \quad \text{ for any $v \in \bR^n {\setminus} \{0\}$,  $A \in \Sym(n)$ } ;$$
\item[ii')] $\eF$ is upper partially elliptic if $$\inf_{\mathsmaller{\a>0}}\,\ul{\eF}(0,v,\a\, v{\otimes}v^{\mathsmaller{T}}-A)<0 \quad \text{ for any $v \in \bR^n {\setminus} \{0\}$,  $A \in \Sym(n)$ } ;$$
\item[iii)] $\eF$ has the lower scaling property if there exists a function $\h: (0,1] \to (0,+\infty)$ such that $$\ul{\eF}(cs,cv,cA)\geq \h(c)\, \ul{\eF}(s,v,A) \quad \text{ for any $c \in (0,1]$, $s \in [-1,0]$, $v \in \bR^n \setminus \{0\}$, $A \in \Sym(n)$ } ;$$
\item[iii')] $\eF$ has the upper scaling property if there exists a function $\h: (0,1] \to (0,+\infty)$ such that $$\ul{\eF}(cs,cv,cA)\leq \h(c)\, \ul{\eF}(s,v,A) \quad \text{ for any $c \in (0,1]$, $s \in [0,1]$, $v \in \bR^n \setminus \{0\}$, $A \in \Sym(n)$ } ;$$
\end{itemize} \smallskip

The most famous example of universal Riemannian operator is the Laplace-Beltrami operator, which is associated to the Euclidean operator $$\ul{\eF}_{\rm \mathsmaller{LB}}: \Sym(n) \to \bR \,\, , \quad \ul{\eF}_{\rm \mathsmaller{LB}}(A) \= -\Tr(A) \,\, .$$ It is proper, positively $1$-homogeneous, Lipschitz continuous and uniformly elliptic with ellipticity constants $\l=\L=1$.

\begin{rem} Another well known example is the Monge-Amp{\eac}re operator, which is associated to $$\ul{\eF}_{\rm \mathsmaller{MA}}: \Sym(n) \to \bR \,\, , \quad \ul{\eF}_{\rm \mathsmaller{MA}}(A)=\det(A) \,\, .$$ However, being the determinant not partially elliptic, it does not fall into the treatment of our analysis. \end{rem}

We are going to list below some other examples of universal Riemannian operators. From now on, for any $A \in \Sym(n)$, we will denote by $\mu_1(A) \leq {\dots} \leq \mu_n(A)$ its ordered eigenvalues.

\subsubsection*{Extremal Pucci operators} \hfill \par

Fix two constants $0<\l\leq\L$ and consider the set $\cB_{\l,\L} \= \big\{B \in \Sym(n): \l I \leq B \leq \L I\big\}$. Then, we define the functions $$\begin{aligned}
\ul{\eM}^+_{\l,\L}: \Sym(n) \to \bR \,\, &, \quad \ul{\eM}^+_{\l,\L}(A) \= \! \sup_{B \in \cB_{\l,\L}} \! \big({-}\Tr(B.A)\big)= -\l \! \sum_{\mathsmaller{\mu_i(A)>0}}\mu_i(A)-\L \! \sum_{\mathsmaller{\mu_i(A)<0}}\mu_i(A) \,\, , \\
\ul{\eM}^-_{\l,\L}: \Sym(n) \to \bR \,\, &, \quad \ul{\eM}^-_{\l,\L}(A) \= \! \inf_{B \in \cB_{\l,\L}} \! \big({-}\Tr(B.A)\big)= -\L \! \sum_{\mathsmaller{\mu_i(A)>0}}\mu_i(A)-\l \! \sum_{\mathsmaller{\mu_i(A)<0}}\mu_i(A) \,\, .
\end{aligned}$$ The associated universal Riemannian operators $\eM^{\pm}_{\l,\L}$ are the so called {\it extremal Pucci's operators}, which are the prototype Hamilton-Jacobi-Bellman operators and, perhaps, the simplest example of Isaacs operators. They are proper, positively $1$-homogeneous, Lipschitz continuous and uniformly elliptic with ellipticity constants $0<\l\leq\L$. Moreover, it is well known that the uniform ellipticity can be characterized by means of the extremal operators $\eM^\pm_{\l,\L}$ via \beq \eM^-_{\l,\L}(x,Q) \leq F(x,s,q,Q)-F(x,s,q,0) \leq \eM^+_{\l,\L}(x,Q) \,\, . \label{extPucci} \eeq (see e.g. \cite[Lemma 2.2]{CC}). Furthermore, by \cite[Lemma 2.10]{CC} \beq \ul{\eM}^-_{\l,\L}(A- \a\, v{\otimes}v^{\mathsmaller{T}}) \geq \ul{\eM}^-_{\l,\L}(A)+\a\l |v|^2 \,\, , \label{pePucci} \eeq we get the following

\begin{lemma} Let $F: \wt{J}^2M \to \bR$ be a proper operator. If $F$ is uniformly elliptic, then it is partially elliptic. \label{uepe} \end{lemma}

\begin{proof} Assume that $F$ is uniformly elliptic with ellipticity constants $0<\l \leq \L$ and fix $x_{\zero} \in M$. Then, for any $x \in B(x_{\zero},R(x_{\zero}))$, since $(\n f_{x_{\zero}})(x)\neq 0_x$, by \eqref{extPucci} and \eqref{pePucci} we get $$\begin{aligned}
F\big(x,0,(Df_{x_{\zero}})&(x),(D^2f_{x_{\zero}})(x){-}\a\,(Df_{x_{\zero}})(x){\otimes}(Df_{x_{\zero}})(x)\big) \\
&\geq F\big(x,0,(Df_{x_{\zero}})(x),0\big) + \eM^-_{\l,\L}\big(x,(D^2f_{x_{\zero}})(x){-}\a\,(Df_{x_{\zero}})(x){\otimes}(Df_{x_{\zero}})(x)\big) \\
&\geq F\big(x,0,(Df_{x_{\zero}})(x),0\big) + \eM^-_{\l,\L}\big(x,(D^2f_{x_{\zero}})(x)\big)+\a\l \big|(\n f_{x_{\zero}})(x)\big|^2
\end{aligned}$$ and therefore $F$ is lower partially elliptic by taking $\a$ large enough. On the other hand, the operator $$F^-: \wt{J}^2M \to \bR \,\, , \quad F^-(x,s,q,Q) \= -F(x,-s,-q,-Q)$$ is uniformly elliptic as well. Hence, by means of our previous computation, we obtain that $F^-$ is lower partially elliptic, which is equivalent to say that $F$ is upper partially elliptic. \end{proof}

\subsubsection*{Pucci's operators from \cite{Pucci66}} \hfill \par

Fix a constant $0<\a\leq\tfrac1{n}$ and consider the set $\tilde{\cB}_{\a} \= \big\{B \in \Sym(n): B \geq \a I_n \, , \,\, \Tr(B)=1 \big\}$. Then, we define the functions $$\begin{aligned}
\ul{\eP}^+_{\a}: \Sym(n) \to \bR \,\, &, \quad \ul{\eP}^+_{\a}(A) \= \! \sup_{B \in \tilde{\cB}_{\a}} \! \big({-}\Tr(B.A)\big) = -\a\Tr(A)-(1-n\a)\mu_1(A) \,\, , \\
\ul{\eP}^-_{\a}: \Sym(n) \to \bR \,\, &, \quad \ul{\eP}^-_{\a}(A) \= \! \inf_{B \in \tilde{\cB}_{\a}} \! \big({-}\Tr(B.A)\big) = -\a\Tr(A)-(1-n\a)\mu_n(A) \,\, .
\end{aligned}$$ We call the associated universal Riemannian operators $\eP^{\pm}_{\a}$ {\it original Pucci's operators}, as they were firstly introduced in the Euclidean case by Pucci in \cite{Pucci66}. They are proper, positively $1$-homogeneous, Lipschitz continuous and uniformly elliptic with ellipticity constants $0<\a \leq 1-(n-1)\a$, but do not allow to characterize uniform ellipticity of fully nonlinear operators as their extremal counterpart $\mathcal{M}^\pm_{\l,\L}$. Then, by means of \eqref{extPucci}, they are related with $\eM^{\pm}_{\l,\L}$ via $$\eM^-_{\a,1-(n-1)\a}(x,Q) \leq \eP^-_{\a}(x,Q) \leq \eP^+_{\a}(x,Q) \leq \eM^+_{\a,1-(n-1)\a}(x,Q) \,\, .$$

\subsubsection*{$p$-Laplacian operator} \hfill \par

Let $1<p<\infty$ and consider the function $$\ul{\eF}_p: \big(\bR^n {\setminus} \{0\}\big) {\times} \Sym(n) \to \bR \,\, , \quad \ul{\eF}_p(v,A) \= -|v|^{p-2} \big(\Tr(A)+(p-2)|v|^{-2}v^{\mathsmaller{T}}.A.v\big) \,\, .$$ The associated universal Riemannian operator $\eF_p$ is the so called {\it $p$-Laplacian}, which normally appears in divergence form as $$\eF_p[u]=-\div\!\big(|\n u|^{p-2}\n u\big) \,\, .$$ It is proper, continuous and positively $(p{-}1)$-homogeneous. Moreover, for any $v \neq 0$ we have $$\begin{aligned}
&\ul{\eF}_p(v,cA) =c\,\ul{\eF}_p(v,A) \quad \text{ for any $c \in \bR$ } , \\
&\ul{\eF}_p(v,A-\a\, v{\otimes}v^{\mathsmaller{T}}) =\ul{\eF}_p(v,A)+\a(p-1)|v|^p \\
\end{aligned}$$ and so $\eF_p$ is partially elliptic, even though it is not uniformly elliptic.

We conclude by saying that the 1-homogeneous version of the $p$-Laplacian $\eF_p$, called {\it game-theoretic $p$-Laplacian}, that is the operator $\eF^{\mathsmaller{G}}_p$ associated to $$\ul{\eF}^{\mathsmaller{G}}_p(v,A) \= |v|^{2-p}\,\ul{\eF}_p(v,A) = -\Tr(A)-(p-2)|v|^{-2}v^{\mathsmaller{T}}.A.v \,\, .$$ is proper, continuous and positively $1$-homogeneous. Moreover, it is easy to check that it is partially elliptic via condition (ii) in Section \ref{sec;notation}.

\subsubsection*{$\infty$-Laplacian operator} \hfill \par

Consider the function $$\ul{\eF}_{\infty}: \bR^n {\oplus} \Sym(n) \to \bR \,\, , \quad \ul{\eF}_{\infty}(v,A) \= - v^{\mathsmaller{T}}.A.v \,\, .$$ The associated universal Riemannian operators $\eF_{\infty}$ is the so called {\it $\infty$-Laplacian}, which is proper, continuous and positively $3$-homogeneous. Since $$\begin{aligned}
&\ul{\eF}_{\infty}(v,cA) =c\,\ul{\eF}_{\infty}(v,A) \quad \text{ for any $c \in \bR$ } , \\
&\ul{\eF}_{\infty}(v,A - \a\, v{\otimes}v^{\mathsmaller{T}}) = \ul{\eF}_{\infty}(v,A)+\a|v|^4 \,\, , \end{aligned}$$ we conclude that $\eF_{\infty}$ is partially elliptic. Similarly, the positively $h$-homogeneous version $\eF^h_{\infty}$ of the $\infty$-Laplacian, which is the operator associated to $$\ul{\eF}^h_{\infty}(v,A) \= |v|^{h-3} \ul{\eF}_{\infty}(v,A) = - |v|^{h-3} v^{\mathsmaller{T}}.A.v \,\, ,$$ is proper, continuous and it is easy to check that it is partially elliptic.

\subsubsection*{Mean curvature operator} \hfill \par

Consider the function $$\ul{\eF}_{\rm{mc}}: \bR^n {\oplus} \Sym(n) \to \bR \,\, , \quad \ul{\eF}_{\rm{mc}}(v,A) \= (1+|v|^2)^{-\frac32}\big(-(1+|v|^2)\Tr(A)+v^{\mathsmaller{T}}.A.v\big) \,\, .$$ The associated universal Riemannian operator $\eF_{\rm{mc}}$ is the so called {\it mean curvature operator}, which can be written in divergence form as $$\eF_{\rm{mc}}[u]=-\div\!\bigg(\frac{\n u}{\sqrt{1+|\n u|^2}}\bigg) \,\, .$$ It is proper and continuous. Moreover, since $$\begin{aligned}
&\ul{\eF}_{\rm{mc}}(v,cA) =c\,\ul{\eF}_{\rm{mc}}(v,A) \quad \text{ for any $c \in \bR$ } , \\
&\ul{\eF}_{\rm{mc}}(v,A - \a\, v{\otimes}v^{\mathsmaller{T}}) = \ul{\eF}_{\rm{mc}}(v,A)+\a|v|^2(1+|v|^2)^{-\frac32} \,\, ,
\end{aligned}$$ it follows that $\eF_{\rm{mc}}$ is partially elliptic.
\medskip

\section{The Strong Maximum Principle} \label{sec;SMP}

Let $(M,g)$ be a Riemannian manifold. Firstly, we are going to prove a technical result which asserts the propagation of maxima on small balls for subsolutions to lower partially elliptic operators on $(M,g)$. Namely

\begin{prop} Let $F: \wt{J}^2M \to \bR$ be a proper, lower semicontinuous operator which is lower partially elliptic \ref{(l.p.e.)} and has the lower scaling property \ref{(l.s.p.)}. Let $\W \subset M$ be a domain and $u \in \USC(\W)$ a viscosity subsolution to $F[u]=0$ on $\W$ which attains a nonnegative maximum at some point $x_{\zero} \in \W$. Then, for any radius $0<r<\tfrac12\min\big\{R(x_{\zero}),\td(x_{\zero},\p\W)\big\}$ and for any point $y \in S(x_{\zero},r)$, the ball $B(y,r)$ contains a point $\tilde{x}=\tilde{x}(y,r)$ with $u(\tilde{x})=u(x_{\zero})$. \label{propmax} \end{prop}

\begin{proof} Assume by contradiction that there exist a (finite) radius $0<r_{\zero}<\tfrac12\min\big\{R(x_{\zero}),\td(x_{\zero},\p\W)\big\}$ and a point $y_{\zero} \in S(x_{\zero},r_{\zero})$ such that $u(x)<u(x_{\zero})$ for any $x \in B(y_{\zero},r_{\zero})$. Let $\a>0$ to be determined, $f_{y_{\zero}}(x) = \tfrac12\td(x,y_{\zero})^2$ as in \eqref{f_x} and consider the function $$h: B(y_{\zero},R(y_{\zero})) \to \bR \,\, , \quad h(x) \= -\exp\big({-}\a f_{y_{\zero}}(x)\big)+\exp\big({-}\a\tfrac{r_{\zero}^2}2\big) \,\, .$$ Notice that, since $x \mapsto R(x)$ is 1-Lipschitz, it follows that $r_{\zero}<R(y_{\zero})$. Indeed (see Section \ref{sec;prel}): \begin{itemize}
\item[$\bcdot$] if $R(x_{\zero})=+\infty$, then $R(y_{\zero})=+\infty$ as well,
\item[$\bcdot$] if $R(x_{\zero})<+\infty$, then $R(y_{\zero})<+\infty$ and $R(y_{\zero}) \geq R(x_{\zero}) - r_{\zero} > \tfrac12{R(x_{\zero})} > r_{\zero}$.\end{itemize} Moreover, from the very definition, it comes that $h(x)=0$ for any $x \in S(y_{\zero},r_{\zero})$ and $-1<h(x)<0$ for any $x \in B(y_{\zero},r_{\zero})$. Letting $c\=\a\exp\big({-}\a\tfrac{r_{\zero}^2}2\big)>0$, we get $$h(x_{\zero})=0 \,\, , \quad (Dh)(x_{\zero})=c\,(D f_{y_{\zero}})(x_{\zero}) \,\, , \quad (D^2h)(x_{\zero})=c\,\Big((D^2f_{y_{\zero}})(x_{\zero})-\a(D f_{y_{\zero}})(x_{\zero}){\otimes}(D f_{y_{\zero}})(x_{\zero})\Big) \,\, .$$ Hence, by the lower partial ellipticity and the lower scaling property, if we choose $\a>0$ big enough, we get $0<c<1$ and \begin{multline}\label{contr} F\big(x_{\zero},h(x_{\zero}),(Dh)(x_{\zero}),(D^2h)(x_{\zero})\big) \geq \\
\geq \h_{y_{\zero}}(c)\, F\big(x_{\zero},0,(D f_{y_{\zero}})(x_{\zero}),(D^2f_{y_{\zero}})(x_{\zero})-\a(D f_{y_{\zero}})(x_{\zero}){\otimes}(D f_{y_{\zero}})(x_{\zero})\big) >0 \,\, .
\end{multline} By the lower semicontinuity of $F$, there exist two constants $0<r<\min\{r_{\zero},R(y_{\zero})-r_{\zero}\big\}$ and $C>0$ such that $F[h](x)\geq C$ for any $x \in U \= B(x_{\zero},r) \cap B(y_{\zero},r_{\zero})$. By the lower scaling property, the function $\e h$ is a strict supersolution of $F[u]=0$ on $U$ for any $0 < \e < 1$. Assume also that $u(x)-u(x_{\zero})\leq 0$ for any $x \in B(x_{\zero},r)$ and recall that $u(x)-u(x_{\zero})<0$ for any $x \in B(y_{\zero},r_{\zero})$. Then, since $h(x)<0$ for any $x \in \p U \cap B(y_{\zero},r_{\zero})$ and $h(x)=0$ for any $x \in \p U \cap S(y_{\zero},r_{\zero})$, we can choose $0<\e_{\zero}<1$ in such a way that $u(x)-u(x_{\zero})\leq \e_{\zero} h(x)$ for any $x \in \p U$. Let us define then $$\psi: \ol{B}(x_{\zero},r) \to \bR \,\, , \quad \psi(x)\=u(x)-\big(u(x_{\zero})+\e_{\zero}h(x)\big) \,\, .$$ We claim that $\psi(x)\leq 0$ for any $x \in \ol{U} \subset \ol{B}(x_{\zero},r)$. In fact, assuming by contradiction that $\max_{x \in \ol{U}}\psi(x)>0$, then there exists $x' \in U$ such that $\psi(x')= \max_{x \in \ol{U}}\psi(x) >0$. But then, since $u(x_{\zero})+\e_{\zero}h$ is of class $\cC^2$, $Dh(x')\neq0$ and $u(x_{\zero})\geq0$, owing also to the fact that $F$ is proper, we get $$F(x',\e h(x'),\e Dh(x'),\e D^2h(x')) \leq F(x',u(x'),\e Dh(x'),\e D^2h(x'))\leq0,$$ which contradicts our previous claim. By construction, this implies that $\psi\leq0$ in $\ol{B}(x_{\zero},r)$. Since $\psi(x_{\zero})=0$, again this implies that $$F(x_{\zero},\e h(x_{\zero}),\e Dh(x_{\zero}),\e D^2h(x_{\zero})) \leq F(x_{\zero},u(x_{\zero}),\e Dh(x_{\zero}),\e D^2h(x_{\zero}))\leq0 \,\, ,$$ a contradiction with \eqref{contr}.
\end{proof}

As a direct corollary, we obtain the following version of the Strong Maximum Principle (SMP for short).

\begin{theorem}[Strong Maximum Principle] Let $F: \wt{J}^2M \to \bR$ be a proper, lower semicontinuous operator which is lower partially elliptic and has the lower scaling property. Let $\W \subset M$ be a domain and $u \in \USC(\W)$ a viscosity subsolution to $F[u]=0$ on $\W$. If $u$ achieves a nonnegative maximum in $\W$, then $u$ is constant. \label{SMP} \end{theorem}

\begin{proof} Let $\tilde{x} \in \W$ be a maximum point for $u$ in $\W$ and set $L\=u(\tilde{x})>0$, $K \= \{x \in \W : u(x)=L\}$. Assume by contradiction that $K \subsetneq \W$. Since $u$ is upper semicontinuous and $\tilde{x}$ is a maximum point, it follows that $K = \{x \in \W : u(x)\geq L\}$ and so it is closed in $\W$. Then, by classical arguments (see e.g. \cite{GT}) we can find a point $y_{\zero} \in \W \setminus K$ such that $\td(y_{\zero},\p\W)>\td(y_{\zero},\p K)=r$. This means that there exists $x_{\zero} \in S(y_{\zero},r)$ such that $u(x_{\zero})=L$, while $u(x)<L$ for any $x \in B(y_{\zero},r)$. However, this is in contradiction with Proposition \ref{propmax}. \end{proof}

\begin{rem}\label{SMPsimp} If the Riemannian manifold $(M,g)$ has nonnegative sectional curvature, then one can replace the condition \ref{(l.p.e.)} in Theorem \ref{SMP} with \begin{enumerate}[label=(l.p.e.'), leftmargin=40pt]
\item\label{(l.p.e.')} there exists a function $\a_{\zero}: M \to [0,+\infty)$ such that $$F(x,0,q,g_x-\a\,q{\otimes}q)>0 \quad \text{ for any $x \in M$, $q \in T_x^*M \setminus\{0\}$, $\a>\a_{\zero}(x)$ } \, .$$ 
\end{enumerate} Indeed, under the hypothesis of nonnegative sectional curvature, by the Rauch Comparison Theorem it follows that $(D^2f_{x_{\zero}})(x) \leq g_x$ for any $x_{\zero} \in M$ and $x \in B(x_{\zero},\inj(x_{\zero}))$ (see e.g. \cite[Lemma 3.1]{C}). This is the same non-totally degeneracy condition in Bardi-Da Lio \cite{BDL1} found in the Euclidean case. \end{rem}

\begin{counterexample}(From \cite[Example 3.17]{BG1} and \cite{KK}) Assume that $(M,g)$ has nonnegative sectional curvature and fix a point $x_{\zero} \in M$. Let us consider the nonlinear PDE $$-\frac{\D u}{1+|\D u|}+f(x)=0 \,\, ,$$ where $f:M \to \bR$ is defined by setting $f(x_{\zero})\=-1$ and $f(x)\=0$ for any $x \in M \setminus\{x_{\zero}\}$. It is straightforward to check that this operator is proper, but the scaling condition and \ref{(l.p.e.')} fail at $x_{\zero}$. In fact, the SMP is violated by the viscosity subsolution $$u: M \to \bR \,\, , \quad u(x)= \begin{cases} 0 & \text{ if } x \neq x_{\zero} \\ 1 & \text{ if } x = x_{\zero} \end{cases} \,\, .$$
\end{counterexample}

We stress now that part of the proof of Proposition \ref{propmax} can be adapted to prove a viscosity version of the Hopf boundary Lemma (see \cite[Appx E]{Ric}  and \cite[Lemma 3.4]{GT} for the corresponding statements in the linear case). Before doing that, we recall that: \begin{enumerate}[label=(i.b.c.), leftmargin=40pt]
\item\label{(i.b.c.)} a domain $\W \subset M$ satisfies the {\it interior ball condition} at a boundary point $x_{\zero} \in \p\W$ if there exist $y \in \W$ and $0<r<R(y)$ such that $B(y,r) \subset \W$ and $S(y,r)\cap \p\W=\{x_{\zero}\}$.
\end{enumerate} Under such hypothesis, for any point $x \in B(y,r)$ there exists a unique unit speed minimizing geodesic $\g_x:[-\d_x,0) \to B(y,r)$ such that $\g_x(-\d_x)=x$ and $\g_x(0)\=\lim_{t \to 0^-}\g_x(t)=x_{\zero}$.

\begin{theorem}[Hopf boundary Lemma]\label{hopf} Let $F: \wt{J}^2M \to \bR$ be a proper, lower semicontinuous operator which is lower partially elliptic and has the lower scaling property. Let $\W \subset M$ be a domain and $u \in \USC(\W)$ a subsolution to $F[u]=0$ on $\W$. Assume that $\W$ satisfies the interior ball condition at some boundary point $x_{\zero} \in \p \W$ and set $u(x_{\zero}) \= \lim_{r \to 0^+} \sup\{u(x): x \in \W \cap B(x_{\zero},r)\}$. If $0\leq u(x_{\zero}) < +\infty$ and $u< u(x_{\zero})$ in $\W$, then for any interior ball $B(y,r) \subset \W$ as in \ref{(i.b.c.)} it holds \beq \liminf_{t \to 0^+} \frac{(u \circ \g_x)(0)-(u \circ \g_x)(-t)}t >0 \quad \text{ for any } x \in B(y_{\zero},r_{\zero}) \,\, , \label{hopfineq} \eeq where $\g_x$ is the unique geodesic as above. \end{theorem}

\begin{proof} Let $B(y,r) \subset \W$ be as in \ref{(i.b.c.)}. Let also $\a>0$, $f_{y}(x) = \tfrac12\td(x,y)^2$ as in \eqref{f_x} and consider the function $$h: B(y,R(y)) \to \bR \,\, , \quad h(x) \= -\exp\big({-}\a f_{y}(x)\big)+\exp\big({-}\a\tfrac{r^2}2\big) \,\, .$$ By repeating the same argument as in the proof of Proposition \ref{propmax}, one can choose $\a>0$ big enough, $0<\tilde{r}< R(y)-r$ small enough and $0<\e<1$ small enough such that $u(x) \leq u(x_{\zero})+\e h(x)$ for any $x \in \ol{U}$, with $U \= B(x_{\zero},\tilde{r}) \cap B(y,r)$. Then, setting $c\=\a\exp\big({-}\a\tfrac{r^2}2\big)$ we get $$\liminf_{t \to 0^+} \frac{(u \circ \g_x)(0)-(u \circ \g_x)(-t)}t \geq \e \diff h(x_{\zero})(\dot{\g}_x(0)) = \e\,c\,g_{x_{\zero}}\big((\n f_{y})(x_{\zero}),\dot{\g}_x(0)\big)>0$$ which concludes the proof. \end{proof}

\begin{rem} In the hypotheses of Theorem \ref{hopf}, if $u \in \cC^1(\bar{\W})$, then \eqref{hopfineq} reads as $$\diff u(x_{\zero})(v)>0 \quad \text{ for any $v \in T_{x_{\zero}}M$ such that $g_{x_{\zero}}\big((\n f_{y})(x_{\zero}),v\big)>0$} \,\, ,$$ which is the classical statement of the Hopf Boundary Lemma (see \cite[Lemma 3.4]{GT} and \cite[Appx E]{Ric}). \end{rem}

Finally, we remark that Theorem \ref{SMP} can be used to derive a version of the Strong Minimum Principle (SmP for short). Indeed 

\begin{theorem}[Strong Minimum Principle] Let $F: \wt{J}^2M \to \bR$ be a proper, upper semicontinuous operator which is upper partially elliptic \ref{(u.p.e.)} and has the upper scaling property \ref{(u.s.p.)}. Let $\W \subset M$ be a domain and $v \in \LSC(\W)$ a viscosity supersolution to $F[u]=0$ on $\W$. If $v$ achieves a nonpositive minimum in $\W$, then $v$ is constant. \label{SmP} \end{theorem}

\begin{proof} Given such a $F: \wt{J}^2M \to \bR$, it is straightforward to realize that the operator $$F^-: \wt{J}^2M \to \bR \,\, , \quad F^-(x,s,q,Q) \= -F(x,-s,-q,-Q)$$ verifies all the hypothesis of Theorem \ref{SMP}. Hence, we get the thesis. \end{proof}

\begin{rem} The proof of Proposition \ref{propmax}, Theorem \ref{SMP} and Theorem \ref{SmP} remain untouched if we replace the conditions \ref{(l.s.p.)} and \ref{(u.s.p.)}, respectively, with the following: \begin{enumerate}[label=(l.s.p.'), leftmargin=40pt]
\item\label{(l.s.p.')} there exists $\hat{F}$ satisfying \ref{(l.s.p.)} and, for any $x_{\zero} \in M$, a third function $\hat{\h}_{x_{\zero}}: (0,1) \to \bR$ such that $\lim_{c \to 0^+}\hat{\h}_{x_{\zero}}(c)=0$ and $$\begin{aligned}
F\big(x,cs,&c(Df_{x_{\zero}})(x),c\big((D^2f_{x_{\zero}})(x){-}\a\,(Df_{x_{\zero}})(x){\otimes}(Df_{x_{\zero}})(x)\big)\big) \geq \\
&\hat{F}\big(x,cs,c(Df_{x_{\zero}})(x),c\big((D^2f_{x_{\zero}})(x){-}\a\,(Df_{x_{\zero}})(x){\otimes}(Df_{x_{\zero}})(x)\big)\big)+\h_{x_{\zero}}(c)\hat{\h}_{x_{\zero}}(c)
\end{aligned}$$ for any $c \in (0,1]$, $s \in [-1,0]$ and $x \in B(x_{\zero},R(x_{\zero}))$;
\end{enumerate} \begin{enumerate}[label=(u.s.p.'), leftmargin=40pt]
\item\label{(u.s.p.')} there exists $\hat{F}$ satisfying \ref{(u.s.p.)} and, for any $x_{\zero} \in M$, a third function $\hat{\h}_{x_{\zero}}: (0,1) \to \bR$ such that $\lim_{c \to 0^+}\hat{\h}_{x_{\zero}}(c)=0$ and $$\begin{aligned}
F\big(x,cs,&c(Df_{x_{\zero}})(x),c\big((D^2f_{x_{\zero}})(x){-}\a\,(Df_{x_{\zero}})(x){\otimes}(Df_{x_{\zero}})(x)\big)\big) \leq \\
&\hat{F}\big(x,cs,c(Df_{x_{\zero}})(x),c\big((D^2f_{x_{\zero}})(x){-}\a\,(Df_{x_{\zero}})(x){\otimes}(Df_{x_{\zero}})(x)\big)\big)+\h_{x_{\zero}}(c)\hat{\h}_{x_{\zero}}(c)
\end{aligned}$$ for any $c \in (0,1]$, $s \in [0,1]$ and $x \in B(x_{\zero},R(x_{\zero}))$.
\end{enumerate} \label{modscal} \end{rem}

We are going to list below some examples that fulfill the hypotheses of our SMP and SmP.

\begin{example} All the universal Riemannian operators listed in Section \ref{URiem}, with the only one exception of the mean curvature operator, satisfy all the hypotheses of Theorem \ref{SMP} and Theorem \ref{SmP}. Hence, both the SMP and the SmP are satisfied.

Actually, we can cover more general universal Riemannian operators, e.g. we can consider $$\ul{\eF}: \bR^n \oplus \Sym(n) \to \bR \,\, , \quad \ul{\eF}(v,A) \= |v|^{\b}\,\ul{\eM}^{\pm}_{\l,\L}(A) \quad \text{ with } \b>1 \,\, .$$ Then, the operator $\eF$ associated to $\ul{\eF}$, which has been studied in \cite{Biri}, is proper and positively $(\b+1)$-homogeneous. Moreover, by \cite[Lemma 2.10]{CC} we get $$\begin{aligned}
&\ul{\eF}(v,A-\a\,v{\otimes}v^{\mathsmaller{T}}) \geq |v|^{\b}\ul{\eM}^{-}_{\l,\L}(A)+\a\l|v|^{2+\b} \,\, ,\\
&\ul{\eF}(v,\a\,v{\otimes}v^{\mathsmaller{T}}-A) \leq -|v|^{\b}\ul{\eM}^{-}_{\l,\L}(A)-\a\l|v|^{2+\b}
\end{aligned}$$ and thus it is partially elliptic. Hence, both the SMP and the SmP hold true. This would extend \cite[Proposition 2.15]{Biri} to Riemannian manifolds. \label{exSMP}\end{example}

\begin{remark} To the best of our knowledge, this is the first SMP for viscosity solutions of the $p$-Laplace equation on Riemannian manifolds. Recently, some results have been established in \cite{BMPR} for $\cC^1$ distributional solutions of equations driven by the $p$-Laplacian on Riemannian manifolds. However, while in the Euclidean case weak (distributional) and viscosity solutions agree (see \cite{Manf}), we are not aware of a similar result in the more general setting of Riemannian manifolds. \end{remark}

\begin{example} We stress that neither \ref{(l.s.p.)} nor \ref{(u.s.p.)} hold true for the mean curvature operator $\eF_{\rm mc}$, but still it fulfills both \ref{(l.s.p.')}, \ref{(u.s.p.')} and so one can obtain the SMP and the SmP in virtue of Remark \ref{modscal} (see also \cite[Sec 3]{BDL1}). The same is true for the capillary surface equation \beq nH\,u(1+|\n u|^2)^{\frac32}-(1+|\n u|^2)\D u +D^2u(\n u,\n u)=0 \,\, , \quad \text{ with $H \in \bR$ } \, . \label{capeq} \eeq Let us notice that the SMP for the mean curvature equation on Riemannian manifolds have been established recently in \cite{BMPR}, while that for the capillary equation appears to be new. \end{example}

\begin{example} Let $\eF$ be one a positively $h$-homogeneous universal Riemannian operator, with $h \in \bR$. Let $a,c : M \to \bR$ be two continuous functions such that $c(x) \geq 0$, $a(x)>0$ for any $x \in M$. Pick a number $k>0$ and assume that either $c \equiv 0$, or $k \geq h$. Then, the operator $$F(x,s,q,Q) \= a(x)\eF(x,q,Q)+c(x)|s|^{k-1}s$$ verifies both the SMP and the SmP. We stress that the capillarity equation \eqref{capeq} is of this form with $\eF=\eF_{\rm mc}$, $k=1$, $a \equiv 1$ and $c \equiv nH$. \end{example}

\begin{example} Let $\eF$ be a {\it Hessian operator}, that is a universal Riemannian operator associated to an Euclidean operator $\ul{\eF}: \Sym(n) \to \bR$. Assume that $\eF$ is continuous, uniformly elliptic and positively $1$-homogeneous. Let also $b \in \cC^0(M,TM)$ a continuous and bounded vector field on $M$. Then, the operators analyzed in \cite{Kim} of the general form $$F[u] \= |\n u|^{p-2}\big(\eF[u]+g(\n u,b)\big)$$ verify both the SMP and the SmP. \end{example}

\medskip

\section{An application: the Strong Comparison Principle} \label{sec;SCP}

Let $(M,g)$ be a Riemannian manifold. In this section, we are going to establish some consequences of the Strong Maximum Principle. \smallskip

%Firstly, we recall the following
%\begin{definition} A fully nonlinear equation {\it $F[u]=0$ satisfies the Strong Comparison Principle (SCP for short) on a domain $\W \subset M$} if $$u \leq v \text{ on $\W$ } \,\, , \quad u(x_{\zero})=v(x_{\zero}) \,\,\text{ for some $x_{\zero} \in \W$} \quad\Longrightarrow\quad u = v \text{ on $\W$ }$$ for any choice of a viscosity subsolution $u \in \USC(\W)$ and a viscosity supersolution $v \in \LSC(\W)$. \end{definition}

In the case of linear equations, being the difference of two solutions a solution as well, it turns out that the (SCP) is equivalent to the (SMP). In the nonlinear case, it is worth observing that the (SCP) implies the (SMP) whenever a constant solves the equation. However, the converse may not hold. In order to have a guideline for the general case, we begin with the following

\begin{lemma} Let $\W \subset M$ be a domain, $x_{\zero} \in \W$ a point and $0<\l\leq\L$ two constants. Consider a viscosity subsolution $u \in \USC(\W)$ and a classical supersolution $v \in \cC^2(\W,\bR)$ of $\eM_{\l,\L}^-[u]=0$ on $\W$ satisfying $u \leq v$. Then, either $u(x) < v(x)$ for any $x \in \W$ or $u\equiv v$. \label{SCP1} \end{lemma}

\begin{proof} It is sufficient to prove that $w \=u-v$ is a subsolution of $\eM_{\l,\L}^-[u]=0$ on $\W$ and then the conclusion follows by Theorem \ref{SMP}. Indeed, since $w\leq0$ by assumptions, if there exists $x_{\zero} \in \W$ such that $w(x_{\zero})=0$, then the SMP would imply $w=0$ on the whole $\W$.

Let us prove now that $w$ is a subsolution to $\eM_{\l,\L}^-[u]=0$ following \cite[Lemma 2.12]{CC}. Let $\psi \in \cC^2(\W,\bR)$ be such that $w{-}\psi \leq 0$ around $x_{\zero}$ and $w(x_{\zero})=\psi(x_{\zero})$. Since $u$ is a viscosity subsolution, we get $$\eM_{\l,\L}^-[\psi+v](x_{\zero})\leq0 \,\, .$$ Moreover, being $\eM_{\l,\L}^-$ sub-additive (see e.g. \cite[Lemma 2.10]{CC}), it holds that $$\eM_{\l,\L}^-[\psi](x_{\zero}) +\eM_{\l,\L}^-[v](x_{\zero})\leq0 \,\, .$$ Since $\eM_{\l,\L}^-[v](x_{\zero})\geq0$ in classical sense by hypothesis, we conclude that $\eM_{\l,\L}^-[\psi](x_{\zero})\leq0$ as desired. \end{proof}

\begin{remark} Let us point out that the proof of Lemma \ref{SCP1} still holds true for any continuous, uniformly elliptic operator $F: J^2M \to \bR$ which is sub-additive, i.e. satisfying $F[\f_1-\f_2] \leq F[\f_1]-F[\f_2]$ for any pair of $\cC^2$ functions $\f_1,\f_2$. Examples of such operators are the Pucci's extremal operators and, more in general, the Hamilton-Jacobi-Bellman operators (see \cite{BC}). \end{remark}

Let us stress now that handling viscosity solutions requires more care and the use of the strong and weak maximum principles, as in \cite{IY,BG1}. Indeed, the fact that $w=u-v$ is a viscosity subsolution to the initial equation is not straightforward in general (see e.g. \cite[Thm 5.3]{CC} for the uniformly elliptic case). Below, we prove the (SCP) by combining the SMP in Theorem \ref{SMP} and the weak comparison principle in \cite[Cor 4.8]{Azagra} following the same lines of \cite[Thm 5.3 and Prop 5.5]{IY}. \smallskip

Before doing that, given a proper and continuous operator $F: J^2M \to \bR$, we recall the following two properties: \begin{enumerate}[label=(i.u.c.), leftmargin=40pt]
\item\label{(i.u.c.)} $F$ is {\it intrinsically uniformly continuous with respect to $x$} if there exists a modulus of continuity $\w$ such that \beq \big|F(y,s,q,Q)-F(x,s,L_{xy}^*q,L_{xy}^*Q)\big| \leq \w\big(\td(x,y)\big) \label{iucw} \eeq for any $x, y \in M$ with $\td(x,y) < \min\{\inj(x),\inj(y)\}$ and $(s,q,Q) \in J^2_yM$;
\end{enumerate} \begin{enumerate}[label=(u.l.p.), leftmargin=40pt]
\item\label{(u.l.p.)} $F$ satisfies the {\it uniform Lipschitz condition} if there exists $C>0$ such that $$|F(x,s_1,q_1,Q)-F(x,s_2,q_2,Q)|\leq C\,\big(|s_1-s_2|+|q_1-q_2|\big)$$ for any $x \in M$ and $(s_1,q_1,Q),(s_2,q_2,Q) \in J^2_xM$.
\end{enumerate}

\begin{remark} Any universal Riemannian operator $\eF$ verifies \eqref{iucw} with $\w=0$. Indeed, the $\fO(n)$-invariance implies that the corresponding Euclidean operator is of the form $\ul{\eF}(s,v,A)= f(s,|v|, \text{eigenvalues of $A$})$. Then, since $|L_{xy}^*q|=|q|$ and $Q, L_{xy}^*Q$ have the same eigenvalues, condition \ref{(i.u.c.)} is fulfilled. \end{remark}

Then, the result reads as follows.

\begin{prop} Assume that $(M,g)$ has non-negative sectional curvature and let $\W \subset M$ be a domain. Let $F: J^2M \to \bR$ be a continuous operator which is uniformly elliptic and satisfying both \ref{(i.u.c.)}, \ref{(u.l.p.)}. Consider a viscosity subsolution $u \in \USC(\W)$ and a viscosity supersolution $v \in \LSC(\W)$ of $F[u]=0$ on $\W$ satisfying $u \leq v$. Then, either $u(x) < v(x)$ for any $x \in \W$ or $u\equiv v$. \label{SCP2} \end{prop}

Before going ahead with the proof of Proposition \ref{SCP2}, we introduce the {\it linearized operator} $$\D_F: J^2M \to \bR \,\, , \quad \D_F(x,s,q,Q) \= \inf \big\{ F(x,s+\tilde{s},q+\tilde{q},Q+\tilde{Q})-F(x,s,q,Q) : (\tilde{s},\tilde{q},\tilde{Q}) \in J^2_xM \big\}$$ and its lower semicontinuous envelope $(\D_F)_*$, which is defined by $$(\D_F)_*(x,s,q,Q) \= \lim_{r \to 0^+}\!\Big(\inf\big\{\D_F(y,s',q',Q'): \td(x,y)+|s-s'|+|q-L_{xy}^*q'|+|Q-L_{xy}^*Q'|<r\big\}\Big) \,\, .$$ Then, we will need the following

\begin{lemma} Assume the same hypothesis of Proposition \ref{SCP2} and set $w \= u-v$. Then, $w$ is a viscosity subsolution of $(\D_F)_*[u]=0$. \label{lemmaDF} \end{lemma}

Notice that Lemma \ref{lemmaDF} can be obtained by following the proof of \cite[Prop 5.5]{IY}, which is stated in the Euclidean case $(M,g)=(\bR^n,g_{\rm flat})$. However, for convenience of the reader, we provide a proof of this result by following the authors' original approach.

\begin{proof}[Proof of Lemma \ref{lemmaDF}] Arguing by contradiction, we assume that there exist a function $\f \in \cC^2(\W,\bR)$ and a point $x_{\zero} \in \W$ such that $w{-}\f < 0$ in $\W\setminus \{x_{\zero}\}$, $w(x_{\zero})=\f(x_{\zero})$,  and $(\D_F)_*[\f](x_{\zero}) > 0$. By lower semicontinuity, we can fix $c>0$ and find a $\d=\d(c)>0$ such that $\d < \inj(x_{\zero})$, $B(x_{\zero},\d) \subset \W$ and $$(\D_F)_*[\f](x) > 2c \quad \text{ for any $x \in B(x_{\zero},\d)$} \, .$$ Set $v_{\f} \= v + \f \in \LSC(\W)$ and observe that $v_{\f}$ is a viscosity supersolution to $F[u]-2c=0$ on the ball $B(x_{\zero},\d)$. Indeed, given $x \in B(x_{\zero},\d)$ and a function $\psi \in \cC^2(B(x_{\zero},\d),\bR)$ such that $v_{\f}{-}\psi \leq 0$ around $x$ and $v_{\f}(x)=\psi(x)$, from the very definition of $v_{\f}$ we get $v{-}(\psi-\f) \leq 0$ around $x$ and $v(x)=\psi(x)-\f(x)$. Hence, being $v$ a viscosity supersolution of $F[u]=0$ on $\W$, from the very definition of $\D_F$ we get $$F[\psi](x) \geq F(x,\psi(x)-\f(x),D\psi(x)-D\f(x),D^2\psi(x)-D^2\f(x))+(\D_F)_*(x,\f(x), D\f(x),D^2\f( x))\geq 2c \,\, .$$

Let us define now $$F_C: J^2M \to \bR \,\, , \quad F_C(x,s,q,Q) \= F(x,s,q,Q) + (C+1)s \,\, .$$ Then, by hypothesis, it follows that: \begin{itemize} 
\item[$\bcdot$] $v_{\f}$ is a viscosity supersolution of $F_C[u]-(C+1)u-2c=0$ on $\W$;
\item[$\bcdot$] $u$ is a viscosity subsolution of $F_C[u]-(C+1)u=0$ on $\W$.
\end{itemize} Since $u\leq v_{\f}$ and $u \in \USC(\W)$, $v_{\f} \in \LSC(\W)$, by using the Hahn-Katetov-Dowker Theorem (see \cite[Thm 3.4.9]{Borwein}), combined with an approximation argument, it follows that there exists a Lipschitz continuous function $h: \ol{B}(x_{\zero},\d) \to \bR$ such that 
$$(C+1)u(x) \leq h(x) \leq (C+1)v_{\f}(x)+c \quad \text{ for any $x \in \ol{B}(x_{\zero},\d)$ .}$$ Therefore: \begin{itemize} 
\item[$\bcdot$] $v_{\f}$ is a viscosity supersolution of $F_C[u]-h(x)-c=0$ on $B(x_{\zero},\d)$;
\item[$\bcdot$] $u$ is a viscosity subsolution of $F_C[u]-h(x)=0$ on $B(x_{\zero},\d)$.
\end{itemize} Since by hypothesis $u-v_{\f}$ has a strict maximum at $x_{\zero}$ and $u(x_{\zero})=v_{\f}(x_{\zero})$, we can take $k>0$ such that
$$u-v_{\f} \leq -k \,\, \text{ on } S(x_{\zero},\d) \quad \text{ and } \quad (2C+1)k \leq c \,\, .$$ Then, we take $v_{\f,k} \= v_{\f}-k$ so that $v_{\f,k} \geq u$ on $S(x_{\zero},\d)$ and $v_{\f,k}(x_{\zero})=u(x_{\zero})-k<u(x_{\zero})$. Note that, since $F$ is Lipschitz with constant $C$, the operator $F_C$ verifies $$F_C(x,r-k,q,Q)\geq F_C(x,r,q,Q)-(2C+1)k \,\, .$$ Therefore, using then the fact that $v_{\f}$ is a viscosity supersolution of $F_C[u]-h(x)-c=0$ on $B(x_{\zero},\d)$ and the inequality $-(2C+1)k\geq -c$, we conclude that the perturbation $v_{\f,k}$ is a viscosity supersolution to $F_C[u]-h(x)=0$ on $B(x_{\zero},\d)$. It is now immediate to observe that the operator $F_C$ verifies the hypotheses of the weak Comparison Principle \cite[Cor 4.8]{Azagra}. In particular, it satisfies condition (1) in \cite[Cor 4.8]{Azagra} (with $\g=1$), since \ref{(u.l.p.)} implies that $r \mapsto F_C(x,r,q,Q)-r$ is non-decreasing on $\bR$. Then, by applying \cite[Cor 4.8]{Azagra}, we can conclude that $v_{\f,k}\geq u$ in $B(x_{\zero},\d)$, which leads to the contradiction $k\leq0$.
\end{proof}

Finally, we are ready to prove Proposition \ref{SCP2}.

\begin{proof}[Proof of Proposition \ref{SCP2}] By Lemma \ref{lemmaDF}, we know that $w$ is a viscosity subsolution of $(\D_F)_*[u]=0$. Then, we consider $G: J^2M \to \bR$ given by $G(x,s,q,Q) \= \eM^+_{\l,\L}(x,Q)+C(|s|+|q|)$ and we observe that $$-(\D_F)_*(x,-s,-q,-Q)\leq G(x,s,q,Q) \quad \text{ for any $(x,s,q,Q) \in J^2M$ . }$$ Hence, it follows that $-w$ is a supersolution of $G[u]=0$. Since $G$ is continuous, uniformly elliptic and positively 1-homogeneous, the result is a consequence of Theorem \ref{SmP}. \end{proof}

\begin{remark} The result holds true for equations driven by the normalized $p$-Laplacian on Riemannian manifold under the assumptions of Proposition \ref{SCP2}. \end{remark}

\begin{remark}\label{seccurv} The assumption on the sectional curvature of $(M,g)$ can be dropped at the expenses of assuming a further uniform continuity assumption in the variable $x$ and $D^2u$ (see \cite[Cor 4.10]{Azagra}). More precisely, Proposition \ref{SCP2} holds true in a general Riemannian manifold, without curvature conditions, if we replace the intrinsic uniform continuity \ref{(i.u.c.)} by: \begin{itemize}[label=(i.u.c.'), leftmargin=40pt]
\item[(i.u.c.')] for any $\e>0$ there exists $\d>0$ such that for any $x,y \in M$ with $\td(x,y) < \min\{\inj(x),\inj(y)\}$ it holds \beq \td(x,y) \leq \d \,\, , \quad -\d I \leq P-L_{xy}^*Q \leq \d I \quad \Longrightarrow \quad \big|F(y,s,q,Q)-F(x,s,L_{xy}^*q,P)\big| \leq \e \label{condAz} \eeq for any $(s,q,P) \in J^2_yM$, $Q \in \Sym^2(T^*_xM)$.
\end{itemize} \end{remark}

\bigskip\bigskip
\font\smallsmc = cmcsc8
\font\smalltt = cmtt8
\font\smallit = cmti8
\hbox{\parindent=0pt\parskip=0pt
\vbox{\baselineskip 9.5 pt \hsize=5truein
\obeylines
{\smallsmc
{\rm \small (Alessandro Goffi)}
\smallskip
Dipartimento di Matematica ``Tullio Levi-Civita'', Universit$\scalefont{0.55}{\text{\Aac}}$ di Padova
Via Trieste 63, 35121 Padova, Italy}
\smallskip
{\smallit E-mail adress}\/: {\smalltt alessandro.goffi@math.unipd.it}
}
}

\bigskip\bigskip
\font\smallsmc = cmcsc8
\font\smalltt = cmtt8
\font\smallit = cmti8
\hbox{\parindent=0pt\parskip=0pt
\vbox{\baselineskip 9.5 pt \hsize=5truein
\obeylines
{\smallsmc
{\rm \small (Francesco Pediconi)}
\smallskip
Dipartimento di Matematica e Informatica ``Ulisse Dini'', Universit$\scalefont{0.55}{\text{\Aac}}$ di Firenze
Viale Morgagni 67/A, 50134 Firenze, Italy}
\smallskip
{\smallit E-mail adress}\/: {\smalltt francesco.pediconi@unifi.it}
}
}

%\bibliography{biblio_smpRiemannian}
%\bibliographystyle{abbrv}
\end{document}